\documentclass[preprint]{elsarticle}
\usepackage{amsmath,amssymb,amsfonts}
\usepackage{graphicx}
\usepackage{amscd}
\usepackage{amsthm}

\newcommand{\Lt}{\mathcal{L}}
\newcommand{\Lnt}{\mathcal{L}_n}

\newcommand{\Drt}{D_{r}X_{t}}
\newcommand{\Drtjj}{D_{r}X^j_{t}}
\newcommand{\Drti}{D_{r}X^i_{t}}

\newcommand{\Dnrti}{D_{r}(X^n_{t})^i}

\newcommand{\Dnrtl}{D_{r} (X^n_{t})^l}

\newcommand{\Dnrtjj}{D_{r} (X^n_{t})^{j}}

\newcommand{\Dnt}{DX^{n}_{t}}
\newcommand{\Dntjj}{D(X^{n}_{t})^j}
\newcommand{\Dnrt}{D_{r}X^n_{t}}
\newcommand{\Dnrs}{D_{r}X^n_{s}}

\newcommand{\R}{\mathbb{R}}

\newcommand{\E}{\mathbb{E}}
\newcommand{\D}{\mathbb{D}}

\newcommand{\Qnt}{Q_{n}(t)}

\newcommand{\Qtij}{Q^{ij}(t)}

\newcommand{\Qntij}{Q_{n}^{ij}(t)}

\numberwithin{equation}{section}

\newtheorem{thm}{Theorem}[section]
\newtheorem{lem}[thm]{Lemma}

\newtheorem{rmk}{Remark}

\newtheorem{hypo}[thm]{Hypothesis}

\title{Integration by Parts Formula and Smoothness of Densities of Solutions to SDE's with Locally Lipschitz Coefficients}
\author[mahdieh]{M. Tahmasebi \corref{cor1}\fnref{fn1}}
\ead{mh_tahmasebi2000@yahoo.com} 

\author[mahdieh, shiva]{S. Zamani}
\ead{zamani@sharif.ir}

\cortext[cor1]{Corresponding author}

\address[mahdieh]{
Department of Applied Mathematics, Faculty of Mathematical Sciences,
Tarbiat Modares University,
P.O. Box 14115-134,  Tehran, Iran}

\address[shiva]{Gratuate School of Management and Economics, Sharif University of Technology,
P.O. Box 11155-9415,  Tehran, Iran}

\begin{document}
\begin{abstract}
In this work we prove the existence of a smooth density for the solution to an SDE with locally Lipschitz and semi-monotone drift, and will derive an exponential decay for this density and all of its derivatives as well.  
Our main tool in this paper is an integration by parts formula for the solution of the mentioned SDE in the Wiener space. We construct an approximating sequence of SDE's with globally Lipschitz drifts
and obtain  a uniform bound for the  integral of their solutions from which we derive the exponential decay for the derivatives of the density of the original SDE.
\end{abstract}

\maketitle
\section{Introduction}
Several studies have focused on the existence of smooth density and Integration by parts formula for the solutions of SDEs.  
Kusuoka and Stroock \cite{Kusu82}  show that an SDE whose coefficients are $C^{\infty}$-globally Lipschitz   and have polynomial growth, has a strong Malliavin differentiable solution of any order. Also, in \cite{Kusuoka85}, assuming some nondegeneracy condition they find an upper bound for all the derivatives of the density of solution, depending on the coefficients of SDE and their derivatives. This nondegeneracy condition 
could be also used to show the absolute continuity of the law of the solution of SDEs with respect to the Lebesgue measure and the smoothness of its density (see e.g. \cite{Nualart06, Jacod87}). In recent years, there were attemps to generalize these results to SDEs with non-globally Lipschitz coefficients.  
For example, 
Y\^{u}ki \cite{Yuki12}  derive some local H\"{o}lder continuity behaviour of the densities of the solutions to SDEs with singular drifts and locallly bounded coefficients. Using a  Fourier transform argument and the Malliavin Calculus, Marco \cite{Marco11}  shows that the solution of an SDE with smooth coefficients for which the derivatives of coefficients are bounded on a domain $D$, has a strong solution with a smooth density on $D$. 
For other references on this subject, we
refer the reader to \cite{Kusuo10, Marco1, Hiraba92}. \\
Assuming the nondegeneracy condition one can derive some integration by parts formula on the Wiener space (see e.g. \cite{Nualart06}).
This formula has many applications for example in financial mathematics. 
It is often of interest to investors to derive an option pricing formula and to know its sensitivity with respect to various parameters. The integration by parts formula obtained from Malliavin calculus can transform the derivative of the option price into weigthed integral of random variables. This gives much more accurate and fast converging numerical solution estimates than obtained by the classical methods \cite{Higa04, Bavou06}. 
The interested reader could see \cite{Alos08, Vive08, Leon07, Marco11, T02, Bally95}.\\
The SDE we consider has not  global Lipschitz 
 coefficients. Such equations mostly come from finance and biology and also dynamical systems and are more challenging when considered on infinite dimensional spaces. (see e.g. \cite{B99, Z95, GM07})\\
In this paper, we consider  an SDE with locally Lipschitz coefficients and uniformly elliptic diffusion. In \cite{tahzam}, we have proved the uniqueness and existence of a solution $X_t$ to this equation. Here we are going to prove the existence of a smooth density for $X_t$, and derive an exponential decay  for the derivatives of this density, at infinity. \\
Since the drift of the SDE is not globally Lipschitz, we will construct a sequence of SDEs with globally Lipschitz drifts whose solutions have uniformly bounded Malliavin derivatives with respect to $n$, and converge to $X_t$ almost everywhere. In this way we can apply the classical Malliavin calculus to these solutions, and by the uniform boundedness of the moments of  inverses of Malliavin covariance matrices and the convergence result we are able to prove an integration by parts formula in the Wiener space. Then we will prove that the densities of laws of the solutions to the constructed sequence of SDEs converge to the solution of the original SDE and derive the exponential decay  of the density of  the solution to it. \\  
The organization of the paper is as follows. In section 2, we present some notions of Malliavin calculus, as stated in \cite{Nualart06}, and by use of them we formulate our main results. In section 3, we will prove the integration by parts formula in the Wiener space. In section 4, we show that the densities of the approximating processes  converge almost surely to the density  of the solution to the original SDE.  Section 5 is devoted to finding a uniform bound for the integrals of the approximating processes and deriving the exponential decay of the density to the solution.  
\section{Formulation of main results}\label{ma}
Let $\Omega$ denote the Wiener space
$C_{_{0}}([0,T];R^d)$. We  furnish $\Omega$ with the
$\parallel.\parallel_{_{\infty}}$-norm making it a (separable)
Banach space. Consider $(\Omega,\mathcal{F},P)$ a complete
probability space, in which $\mathcal{F}$ is generated by the open
sets of the Banach space, $W_{t}$ is a d-dimensional Brownian motion, and $\mathcal{F}_{t}$ is the filtration generated by $W_t$. \\
By $H:=L^2([0,T];\R^d)$ we denote a Hilbert space. For $k, p \geq 1$, denote by $\D^{k,p}$  the domain of the $k$th order Maliavin derivative operator  with respect to the norm

$\begin{aligned} \qquad \parallel F\parallel_{_{k,p}}
                         &=\Big[E\vert F\vert^p+\parallel
                                D^{i_1, \cdots, i_k} F\parallel_{L^p(\Omega; H^{\otimes k})}^p\Big]^\frac{1}{p},
\end{aligned}$\\
and define $\D^{\infty}:= \bigcap_{k,p} \D^{k,p}$.\\
Now consider the following stochastic differential equation\\
\begin{equation}\label{equa}
dX_{t}=b(X_{t}) dt+\sigma(X_{t})dW_{t},  \qquad X_0=x_0.
\end{equation}
where  $b:\R^d \longrightarrow \R^d$  and  $\sigma:\R^d \longrightarrow M_{d\times d}(\R)$ are $C^\infty$ measurable functions. 
The function $\sigma$ is $C^\infty$ and all of its derivatives of order greater or equal 1 are bounded. The function $\sigma$ is globally Lipscitz and consider $k_1$ as its Lipschitz constant.  
$\sigma \sigma^*$ is uniformly continuous with modulus of continuity $\theta(.)$ and there exist $\lambda, \tilde{\lambda} > 0$ such that for every $x,u \in \R^d$
\begin{equation}\label{elliptic}
\lambda {\vert u \vert}^2 \leq \langle \sigma\sigma^*(x)u,u \rangle \leq \tilde{\lambda} {\vert u \vert}^2,
\end{equation}
where $*$ denotes transpose. Also for some positive constant $C_2$
\begin{equation}\label{C2}
\vert  \sigma^*(x)u \vert^2 \leq  C_2 \vert u \vert^2
\end{equation}
Note that $\sigma$ is called uniformly elliptic if the first inequality in (\ref{elliptic}) holds. Notice that in this papre, the function $b$ is not considered globally Lipschitz. \\
In \cite[Theorem 2.2.2 and Corollary 2.2.1]{Nualart06} Nualart
shows that  SDEs which have globally Lipschitz coefficients with 
polynomial growth for all of their derivatives, have the strong unique solutions in $\D^{\infty}$. Also, he stated there the linear equation in which Malliavin derivative satisfy.
 In \cite{tahzam} we have shown that assuming monotonicity for the drifts existence of a unique strong solution $X_t$ in $\D^{\infty}$ to SDE (\ref{equa}). In this paper, we assume that the SDE (\ref{equa}) has a strong unique solution in $\D^\infty$. 
We denote by $\Lt$ the second-order differential operator associated to SDE (\ref{equa}):
\begin{equation*}
\Lt= \frac{1}{2} \sum_{i,j =1}^{d} (\sigma \sigma^*)_{j}^i(x) \partial_{i} \partial_j +  \sum_{i=1}^d b^i(x)\partial_i
\end{equation*}
Consider the following stochastic differential equations\\
\begin{equation}\label{equan}
dX^n_{t}=b_n(X^n_{t}) dt+\sigma(X^n_{t})dW_{t},  \qquad X^n_0=x_0.
\end{equation}
where the functions $b_n$ are globally Lipschitz and all of their derivatives have polynomial growth;  i.e.,
 for each $x \in \R^d$ and each multiindex $\alpha$ with $\vert \alpha \vert = m$, there exist positive constants $q_m$ and $\Gamma_m$ which are independent of $n$ and 
 \begin{equation}\label{roshdbn}
 \vert \partial_\alpha b_n(x) \vert^2 \leq \Gamma_m(1+\vert x \vert^{q_m})
 \end{equation}
 and $b_n(x) =b(x)$ for every $x \in \Omega_n$ where $\Omega = \Omega_n$,\\
As we pointed out, by \cite{Nualart06} there exists unique strong solutions for SDEs (\ref{equan}) and $X_t^n \in \D^\infty$. 
Also, for $r \leq t$,  
\begin{equation}\label{firstderiv}
\Dnrti =\sigma^i(X^n_{r})+\int_{r}^{t}\nabla b_n^i(X^n_{s}).\Dnrs
ds+\int_{r}^{t}\nabla \sigma^i_l(X^n_{s}).\Dnrs dW^l_{s}
\end{equation}
and  for $r > t$, $\Dnrti =0$. Here $u.C$ denotes the product $C^*u$ for a vector $u$ and a matrix $C$, for example $\nabla f(x). \Dnrt= \sum_{l=1}^d \nabla_l f(x) \Dnrtl$. We used the upper index to show a specific row, and the subindex to show a specific column of a matrix.
As before denote the infinitesimal operator associated to these SDE's by $\Lnt$.  

Throughout the paper we assume the following Hypothesis.
\begin{hypo}\label{hypothesis}
 For each $p \geq 1$, there exist some positive constants $c_p$, $\alpha_p$ and $\gamma_p$ such that
\begin{enumerate} 
    \item The sequence $\{X_t^n\}_{n \geq 1}$ converges to $X_t$  in $L^p(\Omega)$.
    \item  
    \begin{equation}\label{karand}
\sup_{n \geq 1}\sup_{0 \leq t \leq T}\E\Big[\Vert \Dnt \Vert_{H}^p\Big] \leq c_{p}.
\end{equation}
   \item 
\begin{equation}\label{Ln}
\Lnt \vert X_t^n - x_0\vert^{p}  \leq \alpha_p \vert X_t^n -x_0\vert^{p} + \gamma_p 
\end{equation}
 \end{enumerate}
 \end{hypo}
\begin{rmk}
Applying It\^o's formula and Inequality (\ref{Ln}) and then using Gronwall's inequality, for maybe rewritting $c_p$, imply that
\begin{equation}\label{ssup}
\sup_{n \geq 1}\sup_{0 \leq t \leq T}\E\Big[\vert X_{t}^n\vert^p\Big] \leq c_{p},
\end{equation}
\end{rmk}

For every $0 \leq t \leq T$, denote the Malliavin covariance matrix of $X_t^n$ and $X_t$  by $Q_n(t)$ and $Q(t)$, respectively.
We know that the diffusion coefficient $\sigma$ in (\ref{equa}) is uniformly elliptic. So that, the H\"{o}rmander condition holds for $\sigma$ and  by Nualart \cite[Theorem 2.3.3 and its proof]{Nualart06} the solutions $X_t^n$ not only have a.s. invertible Malliavin covariance matrices  
and the nondegeneracy condition holds for them, but also they have infinitely differentiable densities. 
Concerning the moments of determinants of inverse Malliavin covariance matrices, assume that
\begin{hypo}\label{karanQn}
There exist some positive constants $c$ and $\lambda_0$ such that for each $p \geq 1$,
\begin{equation}
\sup_{n} \E\Big[det (\Qnt)^{-p}\Big] <
c (\lambda_0 t)^{-d(p-\frac{1}{2})}
\end{equation}
\end{hypo}
In Appendix we will show that this assumption is satisfied by some weak assumptions on $\sigma$.\\
By our assumptions, Hypotheses {\rm \ref{hypothesis}} and {\rm \ref{karanQn}}, we can derive an integration by parts formula in the Wiener space. By use of which and the uniform exponential boundedness of $X_t^n$ with respect to $n$, we will show the exponential decay of the density of the solution to SDE (\ref{equa}).
\section{Integration by parts formula}
In this section, we prove Theorem \ref{IBPg}.  By the integration by parts formula in \cite[Proposition 2.1.4.]{Nualart06} and Hypothesis (\ref{karanQn}), there exists a family of random variables $\{L_{_{\beta}}\}$ depending
on multiindices $\beta$ of length strictly larger than 1 with
coordinates $\beta_{_{j}}\in \{1,...,d\}$,  
such that for every $G \in \D^{\infty}$
\begin{equation}\label{ibpn}
\E[\partial_{_{\alpha}}g(X_t^n)G]=E[g(X_t^n)L_{_{\alpha}}^n(X_t^n,G)],
\end{equation}
and
\begin{equation}\label{karanLn}
\parallel L_{_{\alpha}}^n(X_t^n,G)\parallel_{_{p}} \leq
c_{_{p,q}}\parallel det((\Qnt)^{-1})\parallel_{_{\beta_0}}^m
\parallel DX_t^n\parallel_{_{k,\gamma}}^n
\parallel G
\parallel_{_{k,q}}
\end{equation}
where 
\begin{equation*}
L_{_{(i)}}^n(X_t^n,G)=\sum_{j=1}^{m}{\delta\Big(G(\Qnt^{-1})^{i}_{j}\Dntjj\Big)},
\end{equation*}
and $\delta$ denotes the adjoint of the Mallaivin derivative operator $D$. \\
\begin{thm}\label{IBPg}
Let $g \in C^{m+1}$, and all of its derivatives be bounded and have polynomial growth. If Hypotheses {\rm \ref{hypothesis}} and  {\rm \ref{karanQn}} hold, then for every $G \in \D^{\infty}$ and every multiindex $\alpha=(\alpha_1, \cdots, \alpha_d)$ with $\vert \alpha \vert=m$, there exists a function $H_{\alpha}(X_{t},G)$ such that
\begin{equation}\label{IBPP}
 \E\Big(\partial_{\alpha}g(X_{t})G\Big)=\E\Big(g(X_{t})H_{\alpha}(X_{t},G)\Big),
\end{equation}\\
where $H_{\alpha}=H_{\alpha_{k}}(H_{(\alpha_{1},...,\alpha_{k-1})})$.
\end{thm}
\begin{proof}
Fix $p \geq 1$. Applying (\ref{karand}) and (\ref{karanQn}) to (\ref{karanLn}) we conclude that  $L_{_{\alpha}}^n(X_t^n,G)$ is uniformly bounded and consequently weakly convergent in $L^p(\Omega)$. Let us  denote its limit by $H_{_{\alpha}}(X_t,G)$. 
We are going to show that 
\begin{equation}\label{right}
\E\Big[g(X_{t}^{n})L_{_{\alpha}}^n(X_t^n,G)\Big]  \longrightarrow \E\Big[g(X_{t})H_{_{\alpha}}(X_t,G)\Big] \qquad as ~n \rightarrow \infty.
\end{equation}
By Hypothesis \ref{hypothesis} part (1), there exists a uniformly integrable subsequence of $X_{t}^{n}$
which converges to $X_{t}$ a.s.. Since $\partial_{_{\alpha}}g(.)$ has
polynomial growth, $\partial_{_{\alpha}}g(X_{t}^{n})$s are uniformly integrable and
a.s. convergent to $\partial_{_{\alpha}}g(X_{t})$. Choose $p_1 \geq 1$ such that  $\frac{1}{p}+\frac{1}{p_1}=1$. Thus
\begin{equation}\label{gn}
\partial_{_{\alpha}}g(X_{t}^{n})
\longrightarrow \partial_{_{\alpha}}g(X_{t}) \qquad  in ~L^{p_1}(\Omega),
\end{equation}
Now by Cauchy-Schwarz inequality

$\begin{aligned}
\Big\vert \E\Big[g(X_{t}^{n})L_{_{\alpha}}^n(X_t^n,G)\Big] & -\E\Big[g(X_{t})H_{_{\alpha}}(X_t,G)\Big]\Big\vert \\
&\leq \Big\vert \E\Big[g(X_{t}^{n})L_{_{\alpha}}^n(X_t^n,G)\Big]-\E\Big[g(X_{t})L_{_{\alpha}}^n(X_t^n,G)\Big]\Big\vert \vspace{1cm}\\
 &+ \Big\vert \E\Big[g(X_{t})L_{_{\alpha}}^n(X_t^n,G)\Big]-\E\Big[g(X_{t})H_{_{\alpha}}(X_t,G)\Big]\Big\vert \\
&\leq \E\Big[\vert g(X_{t}^{n})-g(X_{t})\vert^{p_1}\Big]^\frac{1}{p_1}\E\Big[\vert L_{_{\alpha}}^n(X_t^n,G)\vert^p\Big]^\frac{1}{p}\\
&+ \Big\vert\E\Big[g(X_{t})L_{_{\alpha}}^n(X_t^n,G)\Big]-\E\Big[g(X_{t})H_{_{\alpha}}(X_t,G)\Big]\Big\vert ,
\end{aligned}$\\
\\
By (\ref{gn}) and the uniform boundedness of $\Vert L_{_{\alpha}}^n(X_t^n,G)\Vert_p$ the first term in the right hand side of the above inequality tends to $0$. 
Since $g(X_t) \in L^{p_1}(\Omega)$ and $L_{_{\alpha}}^n(X_t^n,G) \rightharpoonup H_{_{\alpha}}(X_t,G)$ weakly in $L^p(\Omega)$, the second term also
tends to $0$ and (\ref{right}) holds.\\
Now, since $\partial_{_{\alpha}}g(X_{t}^{n})$ converges a.s. to $\partial_{_{\alpha}}g(X_{t})$, and $\partial_{_{\alpha}}g \in C^{1}$ is
 a bounded function, by Lebesque's dominated convergence theorem the following holds also true.
\begin{equation}\label{left}
\E\Big[\partial_{_{\alpha}}g(X_t^n)G\Big] \longrightarrow \E\Big[\partial_{_{\alpha}}g(X_t)G\Big]  \qquad as ~n \rightarrow \infty.
\end{equation}
Therefore, letting $n$ tend to $\infty$ in both sides of (\ref{ibpn}) and using (\ref{left}) and (\ref{right}) the integration by parts formula (\ref{IBPg}) results.\\
Notice that by (\ref{IBPP}) for every $k \geq 2$ and every multiindex $\alpha:= (\alpha_1, \cdots, \alpha_k)$, there  exist some functions $H_{\alpha}(X_t,G)$ and $H_{(\alpha_1, \cdots, \alpha_{k-1})}(X_t,G)$ such that 
\begin{equation}\label{ddd}
 \E\Big(\partial_{\alpha}g(X_{t})G\Big)=\E\Big(g(X_{t})H_{\alpha}(X_{t},G)\Big),
\end{equation}
and
\begin{equation*}
 \E\Big(\partial_{(\alpha_1, \cdots, \alpha_{k-1})}g(X_{t})G\Big)=\E\Big(g(X_{t})H_{(\alpha_1, \cdots, \alpha_{k-1})}(X_{t},G)\Big).
\end{equation*}
Also there exist a function $H_{\alpha_k}$ such that 
\\
$\begin{aligned}
 \hspace{2cm}\E\Big(\partial_{(\alpha_1, \cdots, \alpha_{k-1})}\partial_{\alpha_k}g(X_{t})G\Big) &=\E\Big(\partial_{\alpha_k}g(X_{t})H_{(\alpha_1, \cdots, \alpha_{k-1})}(X_{t},G)\Big) \\
 &= \E\Big(g(X_{t})H_{\alpha_k}(X_t, H_{(\alpha_1, \cdots, \alpha_{k-1})}(X_{t},G)\Big)\\
 & = \E\Big(g(X_{t})H_{\alpha_k} H_{(\alpha_1, \cdots, \alpha_{k-1})}(X_{t},G)\Big)
\end{aligned}$\\
\\
Since this equality holds for every $G \in \D^{\infty}$, using (\ref{ddd}) implies that $H_{\alpha}=H_{\alpha_{k}}(H_{(\alpha_{1},...,\alpha_{k-1})})$.
\end{proof}
 \section{Pointwise Convegence of the sequence of densities}
The existence of a density for $X_t$, called $\rho(t,x_0,.)$, is a result of Theorem 2.1.1 in \cite{Nualart06}. But in order to find a $C^{\infty}$ density, note that there exists a sequence $\{\rho^n(t,x_0,.)\}_n$ of $C^{\infty}$ densities associated to $\{X_t^n\}$. Hereby we are going to prove the convergence of $\rho^n(t,x_0,.)$ to $\rho(t,x_0,.)$ in some sense, thereby finding some bounds for the derivatives of $\rho(t,x_0,.)$.\\
To this end,  we will prove the following lemma by which and  Lemma 11.4.1 in \cite{Stroock79} the convergence of densities  would results in $L^1(\Omega)$.  
\begin{lem}\label{rho-n-psi}
If Hypothesis {\rm \ref{karanQn}} holds, then there exists a non-decreasing function $\psi: (0, \infty) \longrightarrow (0, \infty)$, depending only on $d$, $\lambda$ and $\Lambda$ such that $lim_{\epsilon\rightarrow 0}\psi(\epsilon)=0$ and for every $R >0$
\begin{equation*}
\sup_{n \geq 1}\int_{B(0,R)}{\lvert
\rho^n(t,x_0,y+h)-\rho^n(t,x_0,y)}\rvert dy \leq t^{-\nu}\psi(\lvert
h\rvert),
\end{equation*}
where $B(0,R)$ is the open ball with radius $R$ centered at origin.
\end{lem}

\begin{proof} 
This proof is motivated by the proof of Theorem 9.1.15 in \cite{Stroock79}. We will use Girsanov theorem to omit the drift terms of $SDE$s associated to $X_t^n$s.
 \\
Let $ P^n(t,x_0,.)$ be the transition probability associated to $X_t^n$. 
By Girsanov Theorem \cite{Liptser01}, for every $g \in C_0^\infty[0,T] \times \R^d$
\begin{equation*}
\int_{B(0,R)} \Big[\int_0^T g_h(X_s^n)-g(X_s^n) ds \Big]dP^n = \int_{B(0,R)}\Big(\int_0^T g_h(Z_s)-g(Z_s) ds \Big)S_T^n(1)dP
\end{equation*}
where $g_h(x)=g(x+h), ~ x\in \R^d$ and 
\begin{equation*}
S_T^n (c)=[S_t^n(c)]_{_{t=T}} =\Big[ exp\Big(-\frac{c^2}{2} \int_{0}^T \vert u_{t,n}(w) \vert^2 dt+c\int_{0}^{T} \langle u_{t,n}(w) ,  dW_t \rangle \Big)\Big]_{t=T},
\end{equation*}
in which $u_{t,n}$ and $Z_t$ satisfy
\begin{equation}\label{ut}
\sigma(Z_t) u_{t,n}(w) = b_n(Z_t), \quad  and \quad dZ_t = \sigma(Z_t) dW_t .
\end{equation}
By Cauchy-Schwarz inequality 
 
\begin{align}
\int_{B(0,R)} \Big(\int_0^T & g_h(Z_s) -g(Z_s) ds\Big) S_T^n(1)dP   \nonumber\\
& \leq \sqrt{T} \Big[\E  \int_0^T \vert g_h(Z_s)-g(Z_s) \vert^2  ds \Big]^{\frac12} 
     \Big[\int_{Z_t \in B(0,R)} [S_T^n(1)]^2 dp\Big]^{\frac12} .   \label{St}
 \end{align}\\
Now we are going to find some appropriate bound for $\int_{Z_t \in B(0,R)} [S_T^n(1)]^2 dp$  not dependent on $n$. 
Set 
$$dK_t(w) :=\langle u_{t,n}(w) ,  dW_t (w)\rangle$$ 
and use Cauchy-Schwarz inequality and  the submartingale property of $S_t^n(4)$, 
then

$\begin{aligned} 
 \int_{Z_t \in B(0,R)} [S_T^n(1)]^2 dp & \leq  \int_{Z_t \in B(0,R)} \Big[ exp\Big(-4 \int_{0}^T \vert  u_{t,n}(w) \vert^2 dt+2\int_{0}^{T} dK_t \Big) \\
& \qquad \qquad . exp\Big(3 \int_{0}^T \vert  u_{t,n}(w) \vert^2 dt \Big)\Big] dP\nonumber\\
 &\leq \E\Big[ S_T^n(4) \Big]^{\frac12} \int_{Z_t \in B(0,R)} \Big[exp\Big(6 \int_{0}^T \vert  u_{t,n}(w) \vert^2 dt\Big)\Big]^{\frac12} \nonumber\\
  & \leq  \Big[\int_{Z_t \in B(0,R)} exp\Big(6 \int_{0}^T \vert  u_{t,n}(w) \vert^2 dt\Big)dP\Big]^{\frac12}.  \\ 
\end{aligned}$\\
\\
Now, for $n$ large enough and every $x \in B(0,R)$ we know that $b_n(x) =b(x)$. So, multiplying  two sides of (\ref{ut}) by $u_{t,n}(w)$ and using Hypothesis \ref{hypothesis} for $B_T := \sup_{x \in B(0,R)} b(x)$ we have 
 $$\Big[\int_{Z_t \in B(0,R)} exp\Big(6 \int_{0}^T \vert  u_{t,n}(w) \vert^2 dt\Big)dP\Big]^{\frac12} \leq 
 e^{\frac{1}{\lambda}3T(B_T)^2}. $$
Putting this bound into (\ref{St}) and  using Theorem 9.2.12 in \cite{Stroock79}, we will have  
\begin{equation*} 
\int_{B(0,R)} \Big[\int_0^T g_h(X_s^n)-g(X_s^n) ds \Big]dP^n  \leq 4 e^{\frac{1}{\lambda}3T(B_T)^2}\parallel g\parallel t^{-\nu}\psi(\lvert
h\rvert).
\end{equation*}\\
Since for every $g \in C_b([0,T] \times \R^d)$ there exist an increasing sequence $\{g_n\}$ in $ C_0^\infty ([0,T] \times \R^d)$ which converges to $g$,  the monotone convergence theorem completes the proof. 
\end{proof}
Here, we are ready to show the pointwise convergence of densities $\rho^n(t,x_0,.)$ to $\rho(t,x_0,.)$.
\begin{lem}\label{densm}
For every $R > 0$ if $y \in B(0,R)$, then
\begin{equation}
\rho^{n}(t,x_0,y) \longrightarrow \rho(t,x_0,y).
\end{equation}
\end{lem}

\begin{proof}
By Lemma \ref{rho-n-psi}, there exists a non-decreasing function
$\phi:(0, \infty)\longrightarrow (0, \infty)$  not dependent on $n$
such that $lim_{\epsilon\rightarrow 0}\phi(\epsilon)=0$ and
\begin{equation}\label{phim}
\sup_{n \geq 1}\int_{B(0,R)}{\lvert
\rho^n(t,x_0,y+h)-\rho^n(t,x_0,y)}\rvert dy \leq t^{-\nu}\phi(\lvert
h\rvert),
\end{equation}
where $\nu$ depends only on $d$. On the other hand, we have already proved that $X^{n}_{t} \longrightarrow X_{t}$ in $L^2(\Omega)$. So that for every $\psi \in C_b(\R^d)$, 
\begin{equation}\label{rrho}
\int \rho(t,x_0,y)\psi(y)dy =\lim_{n \rightarrow \infty} \int
\rho^n(t,x_0,y)\psi(y) dy
\end{equation}

By (\ref{phim}) and (\ref{rrho}), the requirements of  Lemma 11.4.1 in \cite{Stroock79} hold true, hence $\rho^{n}(t,x_0,.) \longrightarrow \rho(t,x_0,.)$ in $L^1(B(0,R))$. This implies the existence of a subsequence of $\rho^{n}(t,x_0,.) $ which converges a.s. to $\rho(t,x_0,.) $. Finally,  continuity of $\rho(t,x_0,.) $ completes the proof.
\end{proof}
\section{Exponential decay at infinity}
In this section we will find some bounds for the expectations of  solutions. In the next section, we will use these bounds to prove Theorem \ref{densg}.
\begin{lem}\label{exponential}
Assuming the first and the second part of Hypothesis {\rm \ref{hypothesis}}, for every $\zeta>0$ and $q >1$
\begin{equation}\label{expon}
\E\Big[ exp \Big( \sup_{0 \leq t \leq T} [\zeta e^{-\eta t} \vert X_t^n -x_0\vert^{2}] \Big) \Big]  \leq (8C_2\zeta^2+2)exp[-\zeta \frac{\gamma_2}{\alpha_2}],
\end{equation} 
where $\eta = \alpha_2+2C_2 \zeta+ 1/ \zeta.$
\end{lem}
\begin{proof}
Set $$\Gamma (x) := \vert x -x_0 \vert^{2}+\frac{\gamma_2}{\alpha_2},$$ and 
$$Z_t^n:=e^{-\eta t}\vert X_t^n -x_0\vert^{2}+e^{-\eta t}\frac{\gamma_2}{\alpha_2}.$$
By It\^o's formula, 
$$ Z_t^n = \Gamma(x_0) + \int_{0}^{t} e^{-\eta s}[\Lnt \Gamma - \eta \Gamma](X_s^n)ds + \int_{0}^{t} e^{-\eta s}\langle \nabla \Gamma (X_s^n), \sigma(X_s^n) dW_s \rangle.$$
Now consider the function  
$$g_\zeta (r):= exp[\zeta r] , ~~ r > 0$$
again by It\^o's formula,  (\ref{Ln}) and (\ref{C2}),  

$\begin{aligned}
g_\zeta (Z_t^n) & = g_\zeta (Z_0) + \zeta \int_{0}^{t} g_\zeta (Z_s^n)e^{-\eta s}[\Lnt \Gamma - \eta \Gamma](X_s^n)ds \\
& + 2\zeta^2\int_{0}^{t}  g_\zeta (Z_s^n) e^{-2\eta s} \vert  \sigma^*(X_s^n) (X_s^n-x_0) \vert^2  +2\zeta \int_{0}^{t} g_\zeta (Z_s^n) e^{-\eta s} \langle  X_s^n-x_0, \sigma(X_s^n) dW_s \rangle  \\
& \leq g_\zeta (Z_0) + \zeta \int_{0}^{t} g_\zeta (Z_s^n)e^{-\eta s}[\alpha_2 - \eta + 2\zeta C_2] \Gamma(X_s^n)ds \\
& + 2\zeta \int_{0}^{t} g_\zeta (Z_s^n) e^{-\eta s} \langle  X_s^n-x_0, \sigma(X_s^n) dW_s \rangle.  
\end{aligned}$\\
setting $\eta = \alpha_2+2C_2 \zeta  + 1/\zeta$, one has 

$\begin{aligned}
g_\zeta (Z_t^n) & \leq g_\zeta (Z_0) -  \int_{0}^{t} g_\zeta (Z_s^n) Z_s^n ds  + 2\zeta \int_{0}^{t} g_\zeta (Z_s^n) e^{-\eta s} \langle  X_s^n-x_0, \sigma(X_s^n) dW_s \rangle.  
\end{aligned}$\\
Taking expectations, for every $\zeta, t >0$, we have  
\begin{equation}\label{Z0}
\E\Big[g_\zeta (Z_t^n) \Big] + \E\Big[\int_0^t  g_\zeta (Z_s^n)Z_s^n ds\Big] \leq g_\zeta (Z_0).
\end{equation}
On the other hand, by Doob's maximal inequality,  

$\begin{aligned}
\E\Big[ \sup_{t \in [0,T]} \vert g_\zeta (Z_t^n) \vert^2 \Big] & \leq 2 \vert g_\zeta (Z_0) \vert^2 + 8.4.\zeta^2\E\Big[\int_{0}^{T}  e^{-2\eta s} \vert g_\zeta (Z_s^n) \vert^2\vert  \sigma^*(X_s^n) (X_s^n-x_0 )\vert^2 ds\Big] \\
   & \leq 2 \vert g_\zeta (Z_0) \vert^2 + 8.4.\zeta^2C_2\E\Big[\int_{0}^{T}  g_{2\zeta} (Z_s^n) Z_s^n ds\Big], 
\end{aligned}$\\
where for the last inequality we make use of (\ref{C2}). Using (\ref{Z0}), we obtain
\begin{equation*}
\E  \Big[ \sup_{t \in [0,T]} \vert g_\zeta (Z_t^n) \vert^2 \Big]  \leq (32C_2\zeta^2+2)g_{2\zeta} (Z_0).
\end{equation*}
By replacing $\zeta$ by $\zeta/2$ the proof is complete.
\end{proof}
Now, we are going to derive the exponential decay for the density.
\begin{thm}\label{densg}
By Hypothseis {\rm \ref{hypothesis}} and Hypothesis {\rm \ref{karanQn}}, the density of $X_{t}$, $\rho_t(x_0,.)$, is infinitely differentiable and there exist constants $\eta, q, m$ and $c$ such that for every $y \in \R^d$
\begin{equation}\label{boundd}
\max_{\rvert \alpha\rvert \leq n}\Big\rvert\frac{\partial_{\alpha}\rho_t(x_0,y)}{\partial y}\Big\rvert
                   \leq  c \frac{(32C_2+2)^q}{(\lambda_0 t)^{dm(\beta -\frac12)}}exp\Big\{\frac{-2\frac{\gamma_2}{\alpha_2} -2 e^{-\eta t} \vert y -x_0\vert^2}{q} \Big\}
\end{equation}
\end{thm}
\begin{proof}
By Hypothesis \ref{karanQn} we know that the nondegeneracy condition holds for $X_t^n$, so that from Proposition 2.1.5 in \cite{Nualart06} we have
\begin{equation*}
\partial_\alpha \rho^n (t,x_0,y) = (-1)^{\vert \alpha \vert} \E \Big[ \large{1}_{X_t^n > y} L^n_\alpha L^n_{(1,2, \cdots, d)}(X_t^n,1)  \Big] 
\end{equation*}
for every multiindex $\alpha \in \{1, \cdots, d\}^k$ and $y \in \R^d$. 
 Consider the Cauchy-Schwarz inequality 
\begin{equation*}
\vert \partial_\alpha \rho^n (t,x_0,y) \vert \leq  \E \Big[ \large{1}_{X_t^n > y}\Big]^{\frac{1}{q}} \Vert L^n_\alpha L^n_{(1,2, \cdots, d)}(X_t^n,1)  \Vert_p
\end{equation*}
where $\frac{1}{p}+\frac{1}{q}=1.$ 
Applying (\ref{karand}) and (\ref{karanQn}) to (\ref{karanLn}) implies the existence of constants $c_\alpha>0$, $\beta_0 >1$, and $m >1$  not dependent on $n$ such that  
\begin{equation}\label{der-rhoo-n}
\vert \partial_\alpha \rho^n  (t,x_0,y) \vert \leq  \E \Big[ \large{1}_{X_t^n > y}\Big]^{\frac{1}{q}} c_\alpha (\lambda_0 t)^{-dm(\beta_0 -\frac12)}
\end{equation}
Set $\eta := \alpha_2+4C_2  + 1/ 2$, then the Markov's inequality and inequality (\ref{expon}) imply that 
\begin{align}
 P\Big( \large{1}_{X_t^n > y}\Big) & \leq P\Big( exp\{  e^{-\eta t} \vert X_t^n -x_0 \vert^{2} \}
               > exp \{ e^{-\eta t} \vert y -x_0\vert^2 \}\Big)  \nonumber\\
               & \leq  \E\Big[ exp \{2 \sup_{0 \leq t \leq T} [ e^{-\eta t} \vert X_t^n -x_0 \vert^{2}] \} \Big] \Big/   exp \{2 e^{-\eta t} \vert y -x_0\vert^2 \}   \nonumber  \\    
                           & \leq (32C_2+2)exp\Big\{-2\frac{\gamma_2}{\alpha_2} -2 e^{-\eta t} \vert y -x_0 \vert^2 \Big\}.  \label{ccc}
\end{align}
Now,  substituting (\ref{ccc}) in (\ref{der-rhoo-n}) we have
\begin{equation}\label{der-rho-n}
\vert \partial_\alpha \rho^n (t,x_0,y) \vert \leq   c_\alpha (32C_2+2)^q (\lambda_0 t)^{-dm(\beta_0 -\frac12)}exp\Big\{\frac{-2\frac{\gamma_2}{\alpha_2} -2 e^{-\eta t} \vert y -x_0\vert^2}{q} \Big\}.
\end{equation}
The latter inequality implies that the functions ${\partial_{\alpha}\rho^{n} (t,x_0,.)}$ are
equicontniuous on $\R^d$. Thus on every compact subset $V$ in $\R^d$, ${\partial_{\alpha}\rho^{n} (t,x_0,.)}$ contains a
uniformly convergent subsequence, which we denote by
${\partial_{\alpha}\rho^{n} (t,x_0,.)}$, too. This is true especially for
$\lvert \alpha \rvert=2$. 
Hence,  ${(\rho^{n} (t,x_0,.))}$ converges uniformly
on $V$ and ${(\rho^{n} (t,x_0,.))'} \longrightarrow
{(\rho (t,x_0,.))'}$. Using (\ref{der-rho-n}) for every multiindex 
$\alpha$,
\begin{equation*}
{\partial_{\alpha}\rho^{n} (t,x_0,.)} \longrightarrow
\partial_{\alpha}\rho (t,x_0,.)  \qquad  \textmd{uniformly on} ~V.
\end{equation*}
Therefore, $\rho (t,x_0,.)$ is a Schwartz distribution on $\R^d$ and by (\ref{der-rho-n}),  inequality (\ref{boundd}) is satisfied.
\end{proof}
\appendix
\section{The nondegeneracy condition for  $X_t^n$}
\newtheorem{lemm}{Lemma}[section]
\renewcommand*{\thesection}{\Alph{section}}
In this section, we will show that the nondegeneracy condition (Hypothesis \ref{karanQn}) holds if the function $\sigma$ depend only on $t$ instead of $x$ and for $0 \leq t \leq T$ the exists some function $\psi(.)$ such that $\int_0^t \vert \psi(s)\vert^2\exp{\{2s\}}ds < \infty$ and for each $x \in \R^d$, 
\begin{equation}\label{ellip}
\langle \sigma(t)\sigma^*(t) x, x \rangle \leq \psi(t) \vert x \vert^2.
\end{equation}
By Theorem 1.9. in \cite{Kusuoka85}), (\ref{firstderiv}) and It\^o's formula, for every $0 \leq i,j \leq d$ we have 
$$d\langle \Drti, \Drtjj \rangle=\langle \Drti, \nabla b^j(X_{t}).\Drt \rangle  + \langle \nabla b^i(X_{t}).\Drt,  \Drtjj \rangle dt := N_t^{ij}(r) dt.$$
By Fubini's theorem, $Q_{ij}(t)$ ($ij$-component of the matrix $Q$),  satisfies
$$\Qtij =\int_{0}^{t}\langle \sigma^i(r), \sigma^j(r) \rangle dr+\int_{0}^{t}\int_{r}^{t}N_s(r) dsdr $$
\begin{lem}
 Assuming Hypothesis {\rm \ref{hypothesis}} and (\ref{ellip}) for a bounded diffusion $\sigma$, there exists a constant $C_{_{Q}}$ not dependent on $n$ such that   
\begin{equation}\label{karanQ}
\sup_{n}\sup_{0 \leq t \leq T} \E\Big[\vert \Qnt \vert^2\Big]  < C_{_{Q}}.
\end{equation}
\end{lem}
\begin{proof}
By Fubini's theorem, the components of $Q_n(t)$ denoted by $Q^{ij}_n(t)$   satisfy
\begin{equation*}
\Qntij = \int_{0}^{t} \langle \sigma^i(r), \sigma^j(r) \rangle dr+\int_{0}^{t}\int_{0}^{s}N^{n,ij}_s(r) drds, \\
\end{equation*}
where 
$$N^{n,ij}_t(r)  := \langle \Dnrti, \nabla b_n^j(X^n_{t}).\Dnrt \rangle + \langle \nabla b_n^i(X^n_{t}).\Dnrt, \Dnrtjj \rangle$$
Hence, by It\^o's formula 

$\begin{aligned}
\frac{d}{dt}\E\Big[\vert \Qntij \vert^2\Big] & =2\E\Big[\Qntij \langle \sigma^i(t), \sigma^j(t) \rangle\Big]+2\E\Big[\Qntij \int_{0}^{t}N^{n,ij}_t(r) dr\Big] \\
& \leq 2\E\Big[\vert \Qntij \vert^2\Big] +\vert \langle \sigma^i(t), \sigma^j(t) \rangle \vert^2+ \E\Big[\vert \int_{0}^{t}N^{n,ij}_t(r) dr \vert^2\Big]  \\
\end{aligned}$\\
Thus, by (\ref{ellip})

\begin{align}
\frac{d}{dt}\sum_{i,j=1}^{d}\E\Big[\vert \Qntij \vert^2\Big] & \leq 2\sum_{i,j=1}^{d}\E\Big[\vert \Qntij \vert^2\Big] +\vert \psi(t) \vert^2 + \sum_{i,j=1}^{d}\E\Big[\vert \int_{0}^{t}N^{n,ij}_t(r) dr \vert^2\Big]. \label{bbb}
 \end{align}\\
Using Young inequality and  (\ref{karand}), (\ref{roshdbn}) and (\ref{ssup}) we obtain some bounds for the last two terms in the right hand side of inequality (\ref{bbb}), namely some constant $c_1$ such that 
\begin{equation}\label{N}
\sup_{n \geq 1}\E\Big[\vert \int_{0}^{t}N^{n,ij}_s(r) ds \vert^2\Big] < c_1.
\end{equation}
Thus by Lemma 1.1 in \cite{Hasminskii12} we obtain the inequality (\ref{karanQ}).
\end{proof} 
Here, we are going to show that the nondegeneracy condition holds for the sequence $\{X_t^n\}$. To this end, 
\begin{lem}\label{nondegenQn}
Assume that Hypothesis {\rm (\ref{hypothesis})} holds and for vert $0 \leq t \leq T$, $\psi(t) > d+c_1d^2$, then there exist some positive constant $\lambda$ such that 
for every $p\geq 1$, and $0 \leq t \leq T$
\begin{equation}\label{karanqn}
\sup_{n} \E\Big[det (\Qnt)^{-p}\Big] <
c_1 2^{d-1}\frac{\Gamma(2p+\frac{4}{d})}{\Gamma(p)}(\lambda t)^{-d(p-\frac{1}{2})-2}
\end{equation}
\end{lem}

\begin{proof}
Let $x_{i} \geq 0$ for $1 \leq i \leq d$ and define $f:
\R^{d^2}\longrightarrow \R$  by
$$f(y_{11},\ldots,y_{dd})=\prod_{1 \leq i,j \leq
d}exp\{-x_{i}y_{ij}x_{j}\}$$
Then $f$ is in
$C_{0}^2(\R^d)$ and we can use Dynkin's formula (see e.g. \cite[p. 120]{oks03}) and
inequalities (\ref{karand}) and (\ref{karanQ}) to derive a suitable 
bound for the expectation of $Y_t^n:= exp{\{-x^T\Qnt x\}}$: 

$\begin{aligned}
\frac{d}{dt}\E\Big[Y_t^n\Big] &= \E \Big[\sum_{i,j} -x_{i}x_{j}Y_t^n\langle \sigma^i(t), \sigma^j(t) \rangle \Big] +\E \Big[\sum_{i,j} -x_{i}x_{j}Y_t^n\int_{0}^{t}N_t^{n,ij}(r) dr \Big]\\
                                               &\leq (-\psi(t) +c_1d^2 )\lvert x \rvert^2 \E\Big[ Y_t^n \Big]
\end{aligned}$\\
where we have used (\ref{ellip}), (\ref{N}) and the non-negative semidefinitness of the covariance matrix. 
Therefore, from Lemma 1.1 in \cite{Hasminskii12} we have:
\begin{center}
 $\E\Big[[Y_t^n \Big] \leq  exp{\{(-\int_0^t{\psi(s)ds}+c_1d^2t) \lvert x\rvert^2\}} \leq exp{\{-td \lvert x\rvert^2\}}$
\end{center}
Finally, by use of Lemma 7-29 in \cite[p. 92]{Jacod87} and the symmetry of the integral of even functions:

$\begin{aligned}
\E\Big[det (Q_n(t))^{-p}\Big] & \leq  c_1\frac{2^d}{\Gamma(p)}\int_{(\R^d)_{+}} \vert x\vert^{d(2p-1)} exp{\{-td \lvert x\rvert^2\}}dx \\
&= c_1 2^{d}\frac{\Gamma(2p)}{\Gamma(p)}\Big(td\Big)^{-d(p-\frac{1}{2})-2}
\end{aligned}$\\
\end{proof}
\begin{rmk}
 If the Malliavin covariance matrix $Q(t)$ is a.s. invertible, then we can derive Lemma {\rm \ref{nondegenQn}} and the nondegeneracy condition for $X_t$. 
\end{rmk}



\end{document}